\documentclass[11pt]{amsart}
\usepackage{newlfont,amsmath,enumerate,amssymb}

\newcommand{\SL}{{\mathrm{SL}}}

\newtheorem{lemma}{Lemma}

\newtheorem{prop}{Proposition}
\newtheorem{thm}{Theorem}
\newtheorem{cor}{Corollary}

\begin{document}


\title[]{A tale of two Hecke algebras} 

\author{Gordan Savin}

\address{Department of
Mathematics, University of Utah, Salt Lake City, UT 84112}
\thanks{Partially supported by NSF grant DMS 0852429.}
\email{savin@math.utah.edu}
\begin{abstract} 
We use Bernstein's presentation of the Iwahori-Matsumoto Hecke algebra to obtain a simple proof 
of the Satake isomorphism and, in the same stroke, compute the center of the Iwahori-Matsumoto Hecke algebra. 
\end{abstract} 

\maketitle

\section{Introduction} 

 Let $G$ be a connected, split, reductive group over a non-archimedean local field  $F$.  
 Fix a maximal split torus  $T$ in $G$.  Then $T$ determines a root system $\Phi$. 
 Let $W$ be the corresponding Weyl group. 
 Let $K$ be a hyper-special maximal compact subgroup of $G$. More precisely, 
 the torus $T$ preserves a unique apartment in the Bruhat-Tits building of $G$, and we pick $K$ 
 to be the stabilizer of a hyper special vertex in the apartment.  Then $T_K=T\cap K$ is a maximal 
 compact subgroup of $T$, and the quotient $X=T/T_K$ is isomorphic to the co-character lattice of $T$. 
 Let $H_K=C_c(K\backslash G/K)$ be the Hecke algebra of $K$-bi-invariant, compactly supported 
 functions on $G$.  Let $B=TN$ be a Borel subgroup containg $T$. 
 Let $f\in C_c(G/K)$. Define $S(f)$, a function on $T/T_K$, by 
 \[ 
 S(f)(t) =\delta^{1/2}(t) \int_N f(tn) ~dn
 \]
 where $\delta$ is the modular character. 
  A famous theorem of Satake \cite{Sa} states that the map $S$ is an isomorphism of $H_K$ and $\mathbb C[X]^W$.

Let  $I\subset K$ be the Iwahori subgroup such that $I\cap  B= K\cap B$.  
Let $H_I=C_c(I\backslash G/I)$ be the Hecke algebra of $I$-bi-invariant, compactly supported 
functions on $G$.  Let $Z_I$ be the center of $H_I$. 
The space $C_c(I \backslash G/K)$ is naturally 
a left $H_I$-module and a right $H_K$-module. Using Bernstein's description of $H_I$ we show, in Theorem \ref{A},  that 
the map $S$ gives an \underline{explicit} isomorphism
\[ 
S : C_c(I \backslash G/K)\rightarrow \mathbb C[X]. 
\] 
Then, as a simple consequence, we prove that  the algebras 
$Z_I$, $H_K $ and $\mathbb C[X]^W $ are isomorphic. 

\section{Some preliminaries} 

The measure on $G$ is normalized so that the volume of $I$ is one. The space $C_c(G)$ of 
locally-constant, compactly supported functions is an algebra with respect to the convolution 
$\ast$ of functions. The unit of the algebra is $H_K$ is denoted by $1_K$. It is a function 
supported on $K$ such that $1_K(k)=\frac{1}{[K:I]}$ for all $k\in K$. 

For every root $\alpha$ we fix a homomorphism $\varphi_{\alpha} : \SL_2(F)\rightarrow G$. 
The co-root $\alpha^{\vee}$ is an element of $X$ represented in $T$ by 
\[ 
\varphi_{\alpha}
\left(\begin{array}{cc}
v & 0 \\
0 & v^{-1} 
\end{array}\right)
\] 
where $v\in F$ has valuation 1. 
 For every $u\in F$, let 
\[ 
x_{\alpha}(u)=\varphi_{\alpha}
\left(\begin{array}{cc} 
1 & u \\
0 &  1 
\end{array}\right). 
\] 
We view the root $\alpha$ as a homomorphism $\alpha : X \rightarrow \mathbb Z$ such that, if $x\in X$ and $t_x\in T$ is a 
representative of $x$, then 
\[ 
t_x x_{\alpha}(u) t_x^{-1}=x_{\alpha} (v u)
\] 
where the valuation of $v$ is $\alpha(x)$.
 We say that $x$ is {\em dominant} if $\alpha(x)\geq 0$ for all positive roots 
$\alpha$. 

\section{Iwahori Matsumoto Hecke algebra} 
 Let $q$ be the order of the residue field of $F$. 
 We summarize first some results of \cite{IM}.  

The $I$-double co-sets in $G$ are parameterized by $\tilde W = N_G(T_K)/T_K$. This group is a 
semi-direct product of the lattice $X$ and the Weyl group $W$. The length function 
$\ell :  \tilde W \rightarrow \mathbb Z$ is defined by 
\[ 
q^{\ell(w)}=[I w I : I].
\] 
Let $T_{w}$ denote the characteristic function of the double coset $IwI$. 
Then $T_{w} T_{v}=T_{wv}$ if and only if $\ell(w)+\ell(v)=\ell(wv)$, 
and $\ell(w)+\ell(v)=\ell(wv)$ if and only if $IwIvI=IwvI$.   

Let $\rho$ be the sum of all positive roots. 
Then $\ell(x)=\rho(x)$  for a dominant $x\in X$. 
It follows that $T_{x}\cdot T_{y}=T_{x+y}$ 
for any two dominant  $x$ and $y$. 
 Any  $x\in X$ can be written as  $x=y-z$ where $y$ and $z$ are two dominant elements in $X$. Following Bernstein, let 
\[ 
\theta_x=q^{(\ell(z)-\ell(y))/2}\cdot T_y T_z^{-1}. 
\]

\begin{prop}\label{relation} 
 Let $x\in X$, and $s\in W$ a reflection corresponding to a simple root $\alpha$. Then 
\[ 
T_s \theta_x- \theta_{s(x)} T_s= (1-q)\frac{\theta_x-\theta_{s(x)}}{1-\theta_{-\alpha^{\vee}}}.  
\] 
\end{prop} 
 Lusztig \cite{Lu} derives this proposition from \cite{IM}. 
 It can be also verified by a direct calculation in $\varphi_{\alpha}(\SL_2(F))$, see \cite{S2}.

\begin{cor} \label{central} Let $x\in X$, and $s\in W$ a simple reflection, as in Proposition \ref{relation}. Then 
\[
T_s (\theta_x + \theta_{s(x)})= (\theta_x + \theta_{s(x)})T_s.
\] 
\end{cor} 

\begin{prop} \label{basis}  (Bernstein's basis) Elements $\theta_x T_w$, where $x\in X$ and $w\in W$, form a basis of $H_I$. 
\end{prop} 
\begin{proof} Since $T_w$, $w\in W$ and $T_x$, with $x$ dominant generate $H_I$, Proposition \ref{relation} implies 
that $\theta_x T_w$ span $H_I$. Thus it remains to prove the linear independence. We follow an argument from \cite{S1}. 
Assume that 
\[ 
\sum_{i,j} c_{i,j} \theta_{x_i} T_{w_j}=0.
\] 
Let $x_0\in X$ be dominant such that $x_0+x_i$ is dominant for all $x_i$ appearing in the sum. Then, after multiplying 
by $\theta_{x_0}$ from the left, 
\[ 
\sum_{i,j} c_{i,j} \theta_{x_0+ x_i} T_{w_j}=0.
\] 
However, if $x$ is dominant then $T_x \cdot T_w=T_{x\cdot w}$. In particular,  $\theta_{x_0+x_i} T_{w_j}$ are linearly independent.  
Thus $c_{i,j}=0$. 
\end{proof}

Let $A$ be the sub algebra of $H_I$ generated by $\theta_x$. Then $ A\cong\mathbb C[X]$ 
via the isomorphism $ \theta_x\mapsto  [x]$. (We shall write an element in the group algebra 
$\mathbb C[X]$ as $\sum_{x\in X} c_x [x]$, where $c_x\in\mathbb C$, in order to distinguish 
$[x-y]$ from $[x]-[y]$.)

\begin{prop} \label{maximal} 
The centralizer of $A$ in $H_I$ is $A$. 
\end{prop} 
 \begin{proof} 
 Let $z\in H_I$. Express $z$ in the Bernstein's basis, and let $\theta_x T_w$ be 
 a term in the expression such that $\ell(w)$ is maximal. If $w=1$, then $z\in A$. Otherwise, 
 there exists $y\in X$ such that $w(y) \neq y$. Now notice that  $\theta_y \cdot \theta_x T_w =\theta_{y+x}$, while 
 \[ 
 \theta_x T_w \cdot \theta_y= \theta_{x+w(y)} T_w + \sum_{z,v} c_{z,v} \theta_{z}T_{v} 
 \] 
 where $\ell(v)<\ell( w)$. As $y-w(y)$ can be made arbitrarily large,  $z$ does not commute with all elements in $A$. 
 \end{proof}



\section{Satake Map} 

We fix the measure on $N$ so that  the volume of $(N\cap K)$ is $[K:I]$.
We identify  $C_c(T/T_K)$ with $\mathbb C[X]$ by $f\mapsto \sum_{x\in X} f(x)[x]$. 
The Satake map $S: C_c(G/K) \rightarrow C_c(T/T_K)=\mathbb C[X]$ is defined by 
\[ 
S (f)(t)=\delta(t)^{1/2} \int_N f(tn) ~dn. 
\]  
It is a formal check (see \cite{Ca})  that $S$, when restricted to $H_K=C_c(K\backslash G/K)$, is a homomorphism and 
the image of $H_K$ is contained in $\mathbb C[X]^W$.

\begin{prop} Let $1_K$ be the identity element of $H_K$. Then 
 $\theta_x\ast 1_K$, $x\in X$, form a basis of 
$C_c(I\backslash G/K)$.
\end{prop} 
\begin{proof} Note that 
$C_c(I\backslash G/K) =C_c(I\backslash G/I)\ast 1_K$. Since $1_K=\frac{1}{[K:I]}\sum_{w\in W} T_w$, the proposition 
 follows  from Proposition \ref{basis}. 
\end{proof} 

\begin{lemma}  \label{key} 
Let $(\pi,V)$ be a smooth $G$-module and $(\pi',V')$ a smooth $B$-module with the trivial action of $N$.  Let 
$S: V \rightarrow V'$ be a map such that $S(\pi(b)v)=\delta^{-1/2}(b) \pi'(b) S(v)$ for every $b\in B$. Then, 
 for every $x\in $X and $v\in V^I$,
\[
S(\pi(\theta_x) v) =\pi'(t_x) v.
\]
\end{lemma}  
This lemma appears in the literature in a special case when $V'=V_N$, the normalized 
Jacquet functor. The proof is the same and therefore omitted.

\begin{thm}\label{A}  The map $S$ induces an isomorphism 
 of left $ A\cong \mathbb C[X]$-modules 
 \[ 
 C_c(I\backslash G/K)\cong \mathbb C[X]
 \]  
 which sends the basis elements $\theta_x \ast 1_K$ to the basis elements $[x]$. 
\end{thm}
\begin{proof} We apply Lemma \ref{key} to $V=C_c(G/K)$, $V'=C_c(T/T_K)$ (considered as left $G$ and 
$T$-modules) and $S$ the Satake map. Then, for every $f\in C_c(I\backslash G/K)$, 
$S(\theta_x \ast f)(t)= S(f)(t_x^{-1} t)$. Thus $S(\theta_x \ast f)=[x]\cdot S(f)$. 
 In particular,  $S(\theta_x \ast 1_K)= [x]\cdot S(1_K)=[x]$, and the theorem follows. 
\end{proof}

Let $Z_I$ be the center of $H_I$.  Let $A^W$ be the span of $\sum_{w\in W} \theta_{w(x)} $ for $x\in X$. 
Corollary \ref{central} implies that $A^W\subseteq Z_I$. 
Let $Z: Z_I \rightarrow H_K$ be a homomorphism defined by  $Z(z) = z\ast 1_K$. 

\begin{thm}  The maps $Z$ and $S$ induce isomorphisms of algebras 
\[ 
A^W \cong Z_I\cong H_K \cong \mathbb C[X]^W.
\] 
\end{thm} 
\begin{proof}  Theorem \ref{A} implies that $S$, restricted to $H_K$, is injective. 
Proposition \ref{maximal} implies that $Z_I\subseteq  A$. This and Theorem \ref{A}  imply  that the map $S\circ Z$ 
is injective. Thus, we have the injections 
\[ 
A^W \subseteq Z_I\subseteq H_K \subseteq \mathbb C[X]^W.
\] 
 Since $(S\circ Z)(\sum_{w\in W} \theta_{w(x)})=S(\sum_{w\in W} \theta_{w(x)}\ast 1_K)=\sum_{w\in W} [w(x)]$, 
 the above injections are isomorphisms. 

\end{proof} 

\vskip 10pt 
\noindent
{\bf Final Remarks.}  A proof of the isomorphism $Z_I\cong \mathbb C[X]^W$ can 
be found  in \cite {Da} and \cite{HKP}. 
 Both approaches are based on the explicit description of the Bernstein component of the category 
of smooth $G$-modules containing the trivial representation. Dat also shows 
that the map $Z$ gives an isomorphism of $Z_I$ and $H_K$.  On the other hand,  
Lusztig \cite{Lu} considers a version 
of the algebra $H_I$ over the ring $\mathbb Z[q^{\pm1/2}]$ where $q$ is considered a formal 
variable. He shows that the center is isomorphic 
to $\mathbb Z[q^{\pm 1/2}][X]^W$ by specializing $q^{1/2}=1$.  No claim is made as to what the center is 
when $q$ is specialized to a power of a prime number.

\end{document}